\documentclass[11pt,letterpaper]{amsart}

\usepackage[applemac]{inputenc}
\usepackage{amsfonts} 
\usepackage{amsmath} 
\usepackage{amsthm} 
\usepackage{xcolor} 
\usepackage{hyperref}
\usepackage{todonotes}
\usepackage{amssymb}
\usepackage[all]{xy}

\makeatletter
\@namedef{subjclassname@2020}{\textup{2020} Mathematics Subject Classification}
\makeatother

\title[]{Weak K\"ahler hyperbolicity is birational}
\author[F. Bei\and B. Claudon \and S. Diverio \and S. Trapani]{Francesco Bei \and Benoît Claudon \and Simone Diverio \and Stefano Trapani}
\address{Francesco Bei and Simone Diverio \\ Dipartimento di Matematica \lq\lq Guido Castelnuovo\rq\rq{} \\ SAPIENZA Universit\`a di Roma \\ Piazzale Aldo Moro 5 \\ I-00185 Roma.}
\email{bei@mat.uniroma1.it \\ diverio@mat.uniroma1.it } 
\address{Benoît Claudon \\ Univ Rennes \\ CNRS \\ IRMAR - UMR 6625 \\ F-35000 Rennes, France \\ and Institut Universitaire de France} 
\email{benoit.claudon@univ-rennes1.fr} 
\address{Stefano Trapani \\  Università di Roma \lq\lq Tor Vergata\rq\rq{} \\ Via della Ricerca Scientifica 1 \\ I-00133 Roma.} 
\email{trapani@mat.uniroma2.it} 
\thanks{The first-named author is partially supported by the \lq\lq Gruppo Nazionale per le Strutture Algebriche, Geometriche e le loro Applicazioni\rq\rq{} of the Istituto Nazionale di Alta Matematica \lq\lq Francesco Severi\rq\rq{}.}
\thanks{The second-named author would like to thank Institut Universitaire de France for providing excellent working conditions. He also benefits from the support of the French government \lq\lq Investissements d’Avenir\rq\rq{} program integrated to France 2030, bearing the following reference ANR-11-LABX-0020-01.}
\thanks{The third-named author is partially supported by the \lq\lq Gruppo Nazionale per le Strutture Algebriche, Geometriche e le loro Applicazioni\rq\rq{} of the Istituto Nazionale di Alta Matematica \lq\lq Francesco Severi\rq\rq{} as well as the \lq\lq SEED PNR\rq\rq{} project of SAPIENZA Universit\`a di Roma}
\thanks{The forth-named author acknowledges the MIUR Excellence Department Project awarded to the Department of Mathematics, University of Rome Tor Vergata, CUP E83C18000100006, as well as the MUR Excellence Department Project MatMod@TOV awarded to the Department of Mathematics, University of Rome Tor Vergata, CUP E83C23000330006}
\date{\today}

\theoremstyle{plain}
\newtheorem{thm}{Theorem}[section]
\newtheorem{cor}[thm]{Corollary}

\newtheorem*{mainthm*}{Main Theorem}
\newtheorem{lem}[thm]{Lemma}
\newtheorem{prop}[thm]{Proposition}
\newtheorem{quest}[thm]{Question}

\theoremstyle{remark}
\newtheorem{rem}[thm]{Remark}

\theoremstyle{definition}
\newtheorem{defn}[thm]{Definition}

\newcommand{\R}{\mathbb{R}}

\newcommand{\B}{\mathbb{B}}
\newcommand{\PP}{\mathbb{P}}

\DeclareMathOperator{\hyp}{hyp}
\DeclareMathOperator{\PU}{PU}
\DeclareMathOperator{\tor}{tor}
\DeclareMathOperator{\BB}{BB}

\begin{document}
\bibliographystyle{amsalpha}
\maketitle
	
\begin{abstract}
We show that a compact Kähler manifold bimeromorphic to a weakly Kähler hyperbolic manifold is weakly Kähler hyperbolic, providing an answer to a problem raised by J. Kollár \cite[Open Problem 18.7]{Kol95}.
\end{abstract}
	
\section{Introduction}

In the chapter devoted to open problems of his book \lq\lq Shafarevic maps and automorphic forms", J. Kollár speaking about Gromov's Kähler hyperbolicity says

\smallskip

\noindent
\emph{\lq\lq From the birational point of view, Kähler hyperbolicity is not natural since it depends on the birational model chosen. It would be desirable to have a birational variant developed. The natural choice seems to be to requires Gromov's condition not for Kähler form but for a degenerate Kähler form.\rq\rq{}}

\smallskip
Let $X$ be a compact Kähler manifold. In \cite{BDET24} the authors introduce a notion called weak Kähler hyperbolicity, mainly in order to study desingularizations of subvarieties of Kähler hyperbolic manifold in the sense of Gromov \cite{Gro91}. It turns out that many key features of Kähler hyperbolic manifolds generalize (in an appropriate way) to weakly Kähler hyperbolic manifolds and, moreover, weak Kähler hyperbolicity seems to be the right notion in order to have birational invariance. 

Given a compact Kähler manifold $X$, define the convex positive cone $\mathcal W_X\subset H^{1,1}(X,\mathbb R)$ to be the set of real cohomology $(1,1)$-classes which are big and nef, and moreover admit a smooth representative which is $\tilde d$-bounded. Here, $\tilde d$-bounded means that such a representative becomes $d$-exact once pulled-back to the universal covering of $X$ and moreover admits a primitive which is bounded with respect to any metric coming from $X$. 

\begin{defn}
A compact Kähler manifold $X$ is said to be \emph{weakly Kähler hyperbolic} if $\mathcal W_X\ne\emptyset$.
\end{defn}

It is proved in \cite{BDET24} that if $\mu\colon X\to Y$ is a modification of compact Kähler manifolds or, more generally, a surjective holomorphic map between compact Kähler manifolds of the same dimension, then the weak Kähler hyperbolicity of $Y$ implies that of $X$. In other words, weak Kähler hyperbolicity pulls back along modifications. 

In this note we prove the analogous property for pushing forwards (we shall actually prove a slightly more general statement, see Remark \ref{rem:genfin}).

\begin{mainthm*}
Let $\mu\colon Z\to X$ be a modification of compact Kähler manifolds. If $Z$ is weakly Kähler hyperbolic, then $X$ is weakly Kähler hyperbolic, too. 
\end{mainthm*}

What we prove in our Main Theorem is that if $\alpha\in\mathcal W_Z$, then the push-forward $\mu_*\alpha$ belongs to $\mathcal W_X$, so that $\mathcal W_X\ne\emptyset$. The key step in the proof shall consist in showing that $\alpha$ is in fact a pull-back of a $\tilde d$-bounded cohomology class on $X$, and this is obtained in a quite surprising fashion, combining techniques from topology and complex geometry.

\smallskip

Being able to push-forward through a modification the property of being weakly Kähler hyperbolic enables us to finally treat the behaviour of weak Kähler hyperbolicity with respect to meromorphic mappings. The following corollary provides the announced answer to \cite[Open Problem 18.7]{Kol95}.

\begin{cor}\label{cor:cor1}
If $X$ and $Y$ are bimeromorphic compact Kähler manifolds, then $X$ is weakly Kähler hyperbolic if and only if $Y$ is weakly Kähler hyperbolic.
\end{cor}

Next, this other corollary answers in particular to \cite[Question~2.30]{BDET24}.

\begin{cor}\label{cor:cor2}
Let $f\colon X\dasharrow Y$ be a generically finite dominant meromorphic mapping between compact Kähler manifolds. 

If $Y$ is weakly Kähler hyperbolic, then $X$ is weakly Kähler hyperbolic. Conversely, if $X$ is weakly Kähler hyperbolic and the induced map on the fundamental groups is injective, then $Y$ is weakly Kähler hyperbolic. 
\end{cor}

Recall that the degeneracy set $Z_X\subset X$ of a weakly Kähler hyperbolic manifold is a quantitative measure of the failure of $X$ of being Kähler hyperbolic (cf. \cite[Definition 2.24]{BDET24}). There is a precise relation between $Z_X$, $Z_Y$, and the exceptional locus of $f$ in the case where $f\colon X\to Y$ is a generically finite surjective map inducing an injective morphism on fundamental groups.

\begin{cor}\label{cor:cor3}
Let $f\colon X\to Y$ be a generically finite surjective holomorphic map between weakly Kähler hyperbolic manifolds. Suppose that the induced map on the fundamental groups is injective. Then, 
$$
Z_X=f^{-1}(Z_Y)\cup\operatorname{Exc}(f).
$$
\end{cor}

Finally, in Section \ref{sec:toroidal}, we show that (suitable) toroidal compactifications of quotients of the complex unit ball provide examples of weakly Kähler hyperbolic manifolds which are not bimeromorphic to any Kähler hyperbolic manifold. On the other hand, such quotients also admit a singular compactification (namely, the Baily--Borel compactification) which should be in some natural sense Kähler hyperbolic: indeed each of its real $2$-cohomology classes are $\tilde d$-bounded. It is then natural to ask the following.

\begin{quest}
Is there a good natural notion of singular Kähler hyperbolic variety, in particular with the property that in each birational class of weakly Kähler hyperbolic manifolds there is a possibly singular model which is Kähler hyperbolic? 
\end{quest}     

We plan to address the above question in a future work.

\subsubsection*{Acknowledgements} 
The second-named author warmly thanks Pierre Py for enlightening discussions about hyperbolic groups.

The third-named author would like to warmly thank Gabriele Viaggi and Roberto Frigerio for various topological explanations, as well as Sébastien Boucksom for the content of Remark \ref{rem:pushnef}.

This project started while the second-named author was in Roma for a research stay. He would like to thank the Dipartimento di Matematica Guido Castelnuovo, SAPIENZA Università di Roma, for its hospitality.

\section{Tools}
For the reader convenience, we collect in this section the main tools for the proof of our Main Theorem, as well as for the study of toroidal compactifications.

\subsection{Push-forward and pull-back of big or nef classes}

First, we recall how the bigness of a general real $(1,1)$-class behaves under generically finite holomorphic maps.

\begin{prop}[{\cite[Proposition 4.12]{Bou02}}]\label{prop:genfinbig}
Let $f\colon Y\to X$ be a surjective holomorphic map between compact Kähler manifolds of the same dimension. Let $\alpha$ and $\beta$ be real $(1,1)$-cohomology classes on $X$ and $Y$ respectively. Then, 
\begin{itemize}
\item[(i)] the class $\alpha$ is big if and only if $f^*\alpha$ is big,
\item[(ii)] the class $f_*\beta$ is big if $\beta$ is so.
\end{itemize}
\end{prop}

The following proposition is stated for compact Kähler manifolds, whereas the original statement is in Bott--Chern cohomology for general compact complex manifolds.

\begin{prop}[{\cite[Théorème 1]{Pau98}}]\label{prop:surjnef}
Let $f\colon Y\to X$ be a surjective holomorphic map between compact Kähler manifolds. Let $\alpha$ be a real $(1,1)$-cohomology class on $X$. Then, the class $f^*\alpha$ is nef if and only if $\alpha$ is.
\end{prop}

\begin{rem}\label{rem:pushnef}
Observe that, in the equidimensional case as in Proposition~\ref{prop:genfinbig}, an analogous statement for push-forwards of nef classes is not true in general as soon as $\dim X=\dim Y\ge 3$ (but it is for surfaces). 

Indeed, take any bimeromorphic map $f\colon X\dasharrow Y$ which is an isomorphism in codimension one. Then, if $\alpha$ is a nef class on $X$ the push-forward $f_*\alpha$ is nef in codimension one but never nef unless $f$ is an actual biholomorphism. This can be seen resolving as usual the singularities of the closure of the graph of $f$ in $X\times Y$ to obtain the following diagram
$$
\xymatrix{
& Z \ar[dl]_p\ar[dr]^q& \\
X \ar@{-->}[rr]_f& & Y
}
$$
where $p$ and $q$ are modifications and every divisor contracted by $p$ is also contracted by $q$. 

The same diagram shows that, while $p^*\alpha$ is obviously nef, $q_*(p^*\alpha)=f_*\alpha$ is not nef, as claimed (we refer the reader to S. Boucksom's PhD thesis for more details on this). 
\end{rem}

\subsection{Brunnbauer, Kotschick, and Schönlinner result on hyperbolic classes}
Here we want to state \cite[Theorem 2.4]{BKS24}, which will be the crucial ingredient for the proof of birationality of weak Kähler hyperbolicity. So, let us introduce some terminology following \cite{BKS24}. 

First, for $X$ a topological space, we define the aspherical subspace $V^k_\textrm{asph}(X)$ of the real singular cohomology $H^k(X,\mathbb R)$ to be the set of $k$-cohomology classes whose pull-back to any sphere is zero. 

Next, for $X$ a finite simplicial complex, define the hyperbolic subspace $V^k_\textrm{hyp}(X)$ of the real singular cohomology $H^k(X,\mathbb R)$ to be the set of $k$-cohomology classes for which there exists a representing closed $k$-form $\omega$ whose pull-back to the universal cover has a $d$-primitive which is bounded with respect to some lifted Riemannian metric, \textsl{i.e.} $\omega$ is $\tilde d$-bounded. Here, we are using that there is a natural isomorphism $H^k_\textrm{dR}(X,\mathbb R)\simeq H^k(X,\mathbb R)$ between the de Rham cohomology and the singular cohomology.

\begin{rem}
A $k$-form $\omega$ on a simplicial complex $X$ consists of a smooth $k$-form $\omega_\sigma$ for every simplex $\sigma\subset X$ such that $\omega_\sigma|_\tau\equiv\omega_\tau$ whenever $\tau\subset\sigma$ is a subsimplex. 

Analogously to the definition of differential forms, one defines Riemannian metrics for simplicial complexes as the choice of a Riemannian metric $g_\sigma$ on every simplex $\sigma$ of $X$ such that $g_\sigma|_\tau\equiv g_\tau$ for $\tau\subset\sigma$.

We refer to \cite{Swa75} for more details.
\end{rem}

\begin{rem}
In our terminology, for $X$ a compact Kähler manifold, we thus have that $\mathcal W_X$ equals the intersection of $V^2_\textrm{hyp}(X)$ with the cone of big and nef classes.
\end{rem}

Note that $V^k_\textrm{hyp}(X)\subset V^k_\textrm{asph}(X)$ as soon as $k\ge 2$, since every continuous map from a $k$-dimensional sphere $f\colon S^k\to X$ factorizes through the universal cover of $X$.

Now, take a finite simplicial complex $X$ with fundamental group $G:=\pi_1(X)$ and consider (a model of) the classifying space $EG\to BG$. Given a universal covering $\tilde X\to X$, there is unique (up to homotopy) classifying map of this universal covering
$$
c_{\tilde X,X}\colon X\to BG
$$
such that $\tilde X$ is isomorphic to the pull-back $c_{\tilde X,X}^* EG$ as $G$-principal bundles. 

Finally, following \cite[Subsection 0.2.C]{Gro91}, we want to define the hyperbolic subspace for $BG$ which does not necessarily have any longer the homotopy type of a finite simplicial complex. So, the subspace $V^k_\textrm{hyp}(BG)$ of the real singular cohomology $H^k(BG,\mathbb R)$ will be the set of $k$-cohomology classes whose pull-back to any finite simplicial complex is hyperbolic.

We are now ready to state the following theorem, which is the key for the birationality of weak Kähler hyperbolicity.

\begin{thm}[{\cite[Theorem 2.4]{BKS24}, see also \cite[Theorem 5.1]{Ked09}}]\label{thm:keyhyperbolic}
Let $X,Y$ be two finite simplicial complex (\textsl{e.g.} compact complex manifolds), $c_{\tilde X,X}\colon X\to B\pi$ the classifying map of the universal covering $\tilde X\to X$, and $f\colon Y\to B\pi$ an arbitrary continuous map. 

If $w\in H^2(B\pi,\mathbb R)$ is a cohomology class such that $c_{\tilde X,X}^*w\in V^2_{\hyp}(X)$, then $f^*w\in V^2_{\hyp}(Y)$. 
\end{thm}

\begin{rem}\label{rem:finitelypresented}
Consider a finitely presented group $G$. Such a group can be realized as the fundamental group of a compact, connected, smooth manifold $X$ of dimension 4 (or higher). Such a manifold can always be triangulated, \textsl{i.e.} it is homeomorphic to a finite simplicial complex. 

What the theorem says in particular is that in order to detect which classes $w\in H^2(BG,\mathbb R)$ are hyperbolic, it is sufficient to check that the pull-back of $w$ to $X$ via the classifying map of its universal covering is hyperbolic.
\end{rem}

Here is a fundamental consequence for us, which is mentioned in \cite{BKS24}, and for which we reproduce a proof here.

\begin{cor}\label{cor:keyhyperbolic}
For the hyperbolic subspace in degree $2$ of a finite simplicial complex $X$, we have
$$
c_{\tilde X,X}^*\bigl(V^2_{\hyp}(BG)\bigr)=V^2_{\hyp}(X).
$$ 
\end{cor}

In particular, we see that $V^2_{\hyp}(X)$ depends only on the fundamental group of $X$, since different (that is, relative to different $G$-principal bundle structures) classifying maps of the universal covering yields the same image.

During the proof we shall need the following fact, surely well-known to experts.

\begin{lem}\label{lem:imagec^*}
The pull-back $c_{\tilde X,X}^*\colon H^2(BG,\mathbb R)\hookrightarrow H^2(X,\mathbb R)$ is injective and, more importantly,
$$
c_{\tilde X,X}^*\bigl(H^2(BG,\mathbb R)\bigr)=V^2_\mathrm{asph}(X).
$$ 
\end{lem}

\begin{proof}
First, every class in $H^k(BG,\mathbb R)$, $k\ge 2$, is aspherical. This is because, for $Z$ a (nice) topological space, a class in $H^k(Z,\mathbb R)$ is aspherical if and only if its kernel as a linear functional on $H_k(Z,\mathbb R)$ contains the image of the Hurewicz homomorphism $\pi_k(Z)\to H_k(Z,\mathbb R)$, and by definition we have $\pi_k(BG)=0$ as soon as $k\ge 2$. Moreover, the image of an aspherical class belonging to $V^2_\textrm{asph}(BG)=H^2(BG,\mathbb R)$ via the pull-back is obviously aspherical. 

Now, we can construct a model of $BG$ by attaching cells of dimension $3$ or higher to $X$ to make the universal cover contractible without affecting $\pi_1(X)$. So, we can suppose that $X\subset BG$, the classifying map $c_{\tilde X,X}$ is the inclusion, and moreover $X$ and $BG$ share the same $2$-skeleton. In particular, the relative homology $H_2(BG,X,\mathbb R)$ vanishes and by duality so does $H^2(BG,X,\mathbb R)$. We now form the long exact sequence of the pair $X\subset BG$ in cohomology, to get
$$
0\to H^2(BG,\mathbb R)\to H^2(X,\mathbb R)\overset\partial\to H^3(BG,X,\mathbb R),
$$
where $\partial$ is the connecting homomorphism, which gives injectivity at once. It only remains to prove that if $w\in V^2_\textrm{asph}(X)$, then $\partial (w)=0$. This is seen directly, since $\partial (w)\in H_3(BG,X,\mathbb R)^*$ acts on a relative $3$-cycle $\gamma+C_3(X)$ by $\partial (w)\bigl(\gamma+C_3(X)\bigr)=w(\partial_3\gamma)$, where $\partial_3$ is the boundary operator and $\partial_3\gamma\in Z_2(X)$ is a $2$-cycle in $X$ which, if non trivial, comes from a $3$-cell attached to $X$ so that $\partial_3\gamma$ is homeomorphic to $S^2$. Therefore, $w(\partial_3\gamma)=0$ since $w$ is aspherical.
\end{proof}

\begin{proof}[Proof of Corollary \ref{cor:keyhyperbolic}]
We have by definition $c_{\tilde X,X}^*\bigl(V^2_\textrm{hyp}(BG)\bigr)\subset V^2_\textrm{hyp}(X)$. 
For the other inclusion, if $\alpha\in V^2_\textrm{hyp}(X)$, then $\alpha\in V^2_\textrm{asph}(X)$, and thus by Lemma \ref{lem:imagec^*} there exists $w\in H^2(BG,\mathbb R)$ such that $\alpha=c_{\tilde X,X}^*w$, so that the pull-back $c_{\tilde X,X}^*w$ is hyperbolic. But then, by Theorem \ref{thm:keyhyperbolic}, the pull-back of $w$ to any finite simplicial complex is hyperbolic, \textsl{i.e.} $w\in V^2_\textrm{hyp}(BG)$.
\end{proof}

\subsection{Bounded cohomology and hyperbolic classes}

Before stating the next lemma, let us recall that the cohomology of a group $G$ can be computed by means of the usual (inhomogeneous) cochains complex
$$
C^k(G,\R):=\bigl\{c\colon G^k\to \R\bigr\}
$$
together with the differential
\begin{multline*}
(d\,c)(g_1,\ldots,g_{k+1}):=c(g_2,\ldots,g_{k+1}) \\ + \sum_{i=1}^k (-1)^ic(g_1,\ldots,g_{i-1},g_ig_{i+1},g_{i+2},\ldots,g_{k+1}) + (-1)^{k+1}c(g_1,\ldots,g_k).
\end{multline*}
The subspace of bounded cochains
$$
C^k_b(G,\R):=\bigl\{c\colon G^k\to \R\mid\textrm{$c$ is bounded}\bigr\}
$$
is stable under $d$ and gives rise to the complex of bounded cochains; the corresponding cohomology groups
$$
H^k_b(G,\R):=H^k\bigl(C^\bullet_b(G,\R),d\bigr)
$$
are the so-called groups of bounded cohomology (see \cite{FrigerioBook} for a more in-depth discussion). Let us finally observe that the natural inclusion of complexes
$$
C^\bullet_b(G,\R)\hookrightarrow C^\bullet(G,\R)
$$
induces natural maps in cohomology
$$
H^k_b(G,\R)\to H^k(G,\R)\simeq H^k(BG,\R).
$$
As above, we focus on the degree $2$ case and state the following result, possibly well-known to experts.

\begin{lem}\label{lem:bounded cohomology}
For any finitely presented group $G$ we have an inclusion
$$
\operatorname{Im}\bigl(H^2_b(G,\R)\to H^2(G,\R)\bigr)\subset V^2_{\hyp}(BG).
$$
In particular, if $G$ is Gromov hyperbolic, then we have:
$$
H^2(G,\R)=V^2_{\hyp}(BG).
$$
\end{lem}

\begin{proof}
According to Remark \ref{rem:finitelypresented}, we can pick any compact differentiable manifold $X$ such that $\pi_1(X)\simeq G$ to check that a bounded $2$-class gives rise to a hyperbolic one. Let $c\in C^2_b(G,\R)$ be a bounded cochain that is a cocycle; it satisfies that for all $(g,h,k)\in G^3$
$$
c(g,h)=c(h,k)-c(gh,k)+c(g,hk).
$$
We can rewrite it in the following form: for all $(g,h,\gamma)\in G^3$
\begin{equation}\label{eq:cocyle}
c(\gamma^{-1}g,g^{-1}h)=c(g,g^{-1}h)-c(\gamma,\gamma^{-1}h)+c(\gamma,\gamma^{-1}g).
\end{equation}
We now explain how to construct an exact $2$-form $\alpha$ defined on the universal covering $\tilde{X}$ which has a bounded primitive, is invariant under the action of the group $G=\pi_1(X)$, and corresponds to the pull-back to $\tilde X$ of the cochain $c$. These properties together ensure that $\alpha$ descends to $X$ and gives a hyperbolic class in $H^2(X,\R)$ which is the pull-back of $c$.

To do so, let us fix a finite open cover $\{V_i\}_{i\in I}$ of $X$ with $V_i$ simply connected. It follows that the $V_i$'s are evenly covered by the universal covering map $p\colon\tilde{X}\to X$. We can thus choose one connected component $U_i$ of $p^{-1}(V_i)$ and get
$$
p^{-1}(V_i)=\bigcup_{g\in G}gU_i.
$$
Finally, let us denote by $\{\chi_i\}_{i\in I}$ a partition of unity subordinate to the open covering $\{V_i\}_{i\in I}$; it induces a partition of unity on $\tilde{X}$, to be denoted by $\{\chi_{i,g}\}_{(i,g)\in I\times G}$, subordinated to the open covering $\{gU_i\}_{(i,g)\in I\times G}$. The function $\chi_{i,g}$ is supported in $gU_i$ and is nothing but $\chi_{i}\circ p|_{gU_i}$. This family enjoys the obvious equivariance property that for all $(\gamma,g)\in G^2$ and for all $i\in I$
\begin{equation}\label{eq:equivariance-partition}
\gamma^*\chi_{i,g}=\chi_{i,\gamma g}.
\end{equation}
With this in hand, let us consider
$$
\alpha:=\sum_{\substack{(i,j)\in I^2 \\ (g,h)\in G^2}} c(g,g^{-1}h)\,d\chi_{i,g}\wedge d\chi_{j,h}
=d\biggl(\sum_{\substack{(i,j)\in I^2 \\ (g,h)\in G^2}} c(g,g^{-1}h)\,\chi_{i,g}\wedge d\chi_{j,h}\biggr).
$$
The last equality implies easily that $\alpha$ has a bounded primitive since $c$ is a bounded cochain, the $\chi_{i,g}$'s come from downstairs and $I$ is finite. Now, let us check that $\alpha$ is invariant under the action of $\gamma\in G$. We have
\begin{align}
\gamma^*\alpha&=\sum_{\substack{(i,j)\in I^2 \\ (g,h)\in G^2}} c(g,g^{-1}h)\,\gamma^*\left(d\chi_{i,g}\wedge d\chi_{j,h}\right) \nonumber \\
&=\sum_{\substack{(i,j)\in I^2 \\ (g,h)\in G^2}} c(g,g^{-1}h)\,d\chi_{i,\gamma g}\wedge d\chi_{j,\gamma h} \label{eq:computation0} \\
&=\sum_{\substack{(i,j)\in I^2 \\ (g,h)\in G^2}} c(\gamma^{-1}g,g^{-1}h)\,d\chi_{i,g}\wedge d\chi_{j,h} \label{eq:computation1} \\
&=\sum_{\substack{(i,j)\in I^2 \\ (g,h)\in G^2}} \bigl(c(g,g^{-1}h)-c(\gamma,\gamma^{-1}h)+c(\gamma,\gamma^{-1}g)\bigr)\,d\chi_{i,g}\wedge d\chi_{j,h} \label{eq:computation2} \\
&=\alpha. \nonumber
\end{align}
Indeed, the equality \eqref{eq:computation0} is just the equivariance property \eqref{eq:equivariance-partition}, equality \eqref{eq:computation1} is a change of variables, and equality \eqref{eq:computation2} is nothing but the cocycle property \eqref{eq:cocyle}. To see that the remaining terms of the sum vanish, it suffices to observe that we can rewrite the corresponding sum as
\begin{multline*}
\sum_{\substack{(i,j)\in I^2 \\ (g,h)\in G^2}} c(\gamma,\gamma^{-1}h)\,d\chi_{i,g}\wedge d\chi_{j,h} \\
=-\sum_{(j,h)\in I\times G}c(\gamma,\gamma^{-1}h)\,d\chi_{j,h}\wedge d\biggl(\sum_{(i,g)\in I\times G} \chi_{i,g}\biggr)=0
\end{multline*}
since the sum in parenthesis is identically equal to $1$.

The last assertion concerning Gromov hyperbolic groups is a consequence of \cite[Theorem 11]{Min01}: for such a group the natural maps
$$
H^k_b(G,\R)\to H^k(G,\R)
$$
are surjective for any $k\ge 2$.
\end{proof}

\section{Proof of Main Theorem}\label{sect:proofMain}

With all the tools previously developed, we are now in a very good position to prove our Main Theorem.

So, let $\mu\colon Z\to X$ be a modification of compact Kähler manifolds, and take $\alpha\in\mathcal W_Z$. We want to show that $\mu_*\alpha$ is in $\mathcal W_X$. 

First of all, since $\alpha$ is a big class and $\mu$ is surjective between equidimensional manifolds, Proposition \ref{prop:genfinbig} gives that $\mu_*\alpha$ is a big class. 

To show the nefness of $\mu_*\alpha$ is subtler, since in general it is false as we observed in Remark \ref{rem:pushnef}. This will be achieved together with $\tilde d$-boundedness of $\mu_*\alpha$ thanks to the topological description of classes which are $\tilde d$-bounded by Brunnbauer, Kotschick, and Schönlinner, as follows. 

Let $G=\pi_1(X)$, and fix a universal covering $\tilde X$ of $X$. Consider the diagram
$$
\xymatrix{
&  & BG \\
Z \ar[rru]^{c_{\tilde X,X}\circ\mu} \ar[dr]_\mu & & \\
& X. \ar[ruu]_{c_{\tilde X,X}}&
}
$$
The map $\mu$ is a modification between smooth manifold, and thus 
$$
\mu_*\colon\pi_1(Z)\to\pi_1(X)=G
$$ 
is an isomorphism (see \textsl{e.g.} \cite[Proposition 2.3]{BP21}). Therefore, pulling-back via $\mu$ the universal covering $\tilde X$ of $X$ gives a universal covering of $\tilde Z\to Z$. By construction, $\tilde Z$ is the pull-back of $EG$ via $c_{\tilde X,X}\circ\mu$, so that 
$$
c_{\tilde X,X}\circ\mu=c_{\tilde Z,Z}
$$ 
is the classifying map of the universal covering $\tilde Z$ as a $G$-principal bundle on $Z$.

Since $\alpha\in V^2_\textrm{hyp}(Z)$, by Corollary \ref{cor:keyhyperbolic}, there exists a (unique) class $w\in V^2_\textrm{hyp}(BG)$ such that $c_{\tilde Z,Z}^*w=\alpha$. Call $\beta:=c_{\tilde X,X}^*w\in V^2_\textrm{hyp}(X)$. Then, we have $\alpha=\mu^*\beta$ by construction and, since $\mu$ is a modification, we also have
$$
\mu_*\alpha=\mu_*\mu^*\beta=\beta,
$$
so that $\beta$ is a real $(1,1)$-class which is, by definition of $V^2_\textrm{hyp}(BG)$, $\tilde d$-bounded. What we have gained is that $\alpha=\mu^*\beta$ is now a pull-back via the surjective holomorphic map $\mu$. Thus, we can apply Proposition \ref{prop:surjnef} to obtain that $\beta$ is nef, since $\alpha$ is nef. We already saw that $\beta$ is big, and therefore $\beta\in\mathcal W_X$, as desired.\qed

\begin{rem}\label{rem:genfin}
What did we really use in the above proof? In order to obtain the bigness of $\mu_*\alpha$, we apply Proposition \ref{prop:genfinbig} which needs $\mu$ to be a surjective holomorphic map between compact Kähler manifolds of the same dimension. Thus, here everything works even if $\mu$ is supposed to be generically finite and surjective. 

Next, in order to obtain $\alpha$ as the pull-back of $\beta$ we need the induced morphism $\mu_*\colon\pi_1(Z)\to\pi_1(X)$ on fundamental groups to be an isomorphism. 

Then, to say that $\mu_*\alpha =\beta$ we use that $\mu$ is a modification. But if $\mu$ is merely supposed to be a generically finite map of degree $m$, then we would have obtained $\mu_*\alpha =m\beta$ so that $\beta=\frac 1m\mu_*\alpha$ would still be a real $\tilde d$-bounded $(1,1)$-class.

Finally, in order to be able to apply Proposition \ref{prop:surjnef} we just need surjectivity of $\mu$. 

Summing up, our Main Theorem is still valid in the more general context of a generically finite surjective holomorphic map $\mu\colon Z\to X$, provided $\mu$ induces an isomorphism at the level of fundamental group.
\end{rem}

\section{Proofs of Corollaries}

In this section we give proofs of the corollaries of our Main Theorem stated in the introduction.

\subsection{Proof of Corollary \ref{cor:cor1}}
Let $f\colon X\dasharrow Y$ be a bimeromorphic mapping. Then, by resolving the singularities of the closure of the graph of $f$ in $X\times Y$, we get a compact Kähler manifold $Z$ together with modifications $p\colon Z\to X$ and $q\colon Z\to Y$ as follows
$$
\xymatrix{
& Z \ar[dl]_p\ar[dr]^q& \\
X \ar@{-->}[rr]_f& & Y.
}
$$
Now, take $\alpha\in\mathcal W_X$. We want to show that $\beta:=f_*\alpha=q_*p^*\alpha$ gives a class in $\mathcal W_Y$. By \cite[Proposition 2.29]{BDET24} $\tilde\alpha:=p^*\alpha$ is a weakly Kähler hyperbolic class in $\mathcal W_Z$, therefore $q_*\tilde\alpha$ is a weakly Kähler hyperbolic class in $\mathcal W_Y$, by our Main Theorem. \qed

\subsection{Proof of Corollary \ref{cor:cor2}}
Let $f\colon X\dasharrow Y$ be a generically finite dominant meromorphic mapping. Then, by resolving the singularities of the closure of the graph of $f$ in $X\times Y$, we get a compact Kähler manifold $Z$ together with a modification $p\colon Z\to X$ and, this time, a generically finite surjective holomorphic map $q\colon Z\to Y$ as follows
$$
\xymatrix{
& Z \ar[dl]_p\ar[dr]^q& \\
X \ar@{-->}[rr]_f& & Y.
}
$$
Now, take $\alpha\in\mathcal W_Y$. Here, we want to show that $\beta:=p_*q^*\alpha$ gives a class in $\mathcal W_X$. By \cite[Proposition 2.29]{BDET24} $\tilde\alpha:=q^*\alpha$ is a weakly Kähler hyperbolic class in $\mathcal W_Z$, therefore $p_*\tilde\alpha$ is a weakly Kähler hyperbolic class in $\mathcal W_Y$, by our Main Theorem.

\smallskip

Before we start the proof of the converse, let us briefly explain what is the morphism induced on fundamental groups by the meromorphic mapping $f$. It is by definition the composition 
$$
f_*:=q_*\circ(p_*)^{-1}\colon\pi_1(X)\to\pi_1(Z)\to\pi_1(Y),
$$
which can be taken since $p$ is a modification and therefore $p_*$ is an isomorphism. It is straightforward to check that the morphism $f_*$ is independent of the smooth bimeromorphic model $Z$ chosen to desingularize the closure of the graph of $f$, since if one picks two of them, there is always a third dominating both.

\smallskip

Coming back to the proof of the corollary, since by hypothesis $f_*$ is injective, it follows that $q_*\colon\pi_1(Z)\to\pi_1(Y)$ is injective. Let $H\le\pi_1(Y)$ be the isomorphic image of $\pi_1(Z)$ in $\pi_1(Y)$. Take the connected étale covering $\pi\colon W\to Y$ corresponding to this subgroup of $\pi_1(Y)$. By construction we thus can take a lifting $\tilde q\colon Z\to W$ of $q$, as in the diagram below
\begin{equation}\label{diag:genfin}
\xymatrix{
& & & W \ar[ddl]^\pi \\
& Z \ar[dl]_p\ar[dr]^q \ar[urr]^{\tilde q} & &\\
X \ar@{-->}[rr]_f& & Y. &
}
\end{equation}
The lifting $\tilde q$ is again generically finite, and $\dim W=\dim Y=\dim Z$, so that $\tilde q$ must be surjective. A posteriori then $W$ is compact Kähler and $H$ of finite index, so that $\pi$ is a finite étale cover. Summing up, $\tilde q$ is a generically finite surjective map, and again by construction $\tilde q_*\colon\pi_1(Z)\to\pi_1(W)$ is an isomorphism.

To conclude: $X$ weakly Kähler hyperbolic implies $Z$ weakly Kähler hyperbolic by \cite[Proposition 2.29]{BDET24}, $Z$ weakly Kähler hyperbolic implies $W$ weakly Kähler hyperbolic by Remark \ref{rem:genfin}, and $W$ weakly Kähler hyperbolic implies $Y$ weakly Kähler hyperbolic by \cite[Proposition 2.6]{BDET24}. \qed

\subsection{Proof of Corollary \ref{cor:cor3}} 

The morphism $f$ is a particular case of that considered in Corollary \ref{cor:cor2}. Thus, we deduce immdiately that if $Y$ is weakly Kähler hyperbolic so is $X$, and conversely if $X$ is weakly Kähler hyperbolic so is $Y$, whenever the induced map on fundamental groups is injective. 

We now focus on degeneracy sets. Recall that if $X$ is weakly Kähler hyperbolic then
$$
Z_X=\bigcap_{[\mu]\in\mathcal W_X}\operatorname{Null}([\mu]),
$$
where $\operatorname{Null}([\mu])$ is the union of all positive dimensional irreducible subvarieties $Z\subset X$ such that $\int_Z\mu^{\dim Z}=0$. It is a proper Zariski closed set in $X$, and by \cite[Remark 2.25]{BDET24} it is always realized as the null set of some weakly Kähler hyperbolic class in $\mathcal W_X$.

\begin{lem}\label{lem:degbelowcontainedindegabove}
Let $f\colon X\to Y$ be a generically finite surjective holomorphic map between weakly Kähler hyperbolic manifolds and $[\mu]\in\mathcal W_Y$. Then, 
$$
\operatorname{Null}([f^*\mu])=f^{-1}\bigl(\operatorname{Null}([\mu])\bigr)\cup\operatorname{Exc}(f).
$$
\end{lem}

\begin{proof}
If $x\in\operatorname{Exc}(f)$ there exists an irreducible curve $x\in C\subset X$ which is contracted by $f$, so that $f^*[\mu]\cdot C=0$ and thus $x\in\operatorname{Null}([f^*\mu])$. 

If $x\in\operatorname{Null}([f^*\mu])$, let $x\in Z\subset X$ be an irreducible positive dimensional subvariety such that $[f^*\mu]^{\dim Z}\cdot Z=0$. The projection formula gives 
$$
0=[f^*\mu]^{\dim Z}\cdot Z=[\mu]^{\dim Z}\cdot f_*Z.
$$
If $\dim f(Z)<\dim Z$, then $Z\subset\operatorname{Exc}(f)$, and thus $x\in\operatorname{Exc}(f)$; otherwise $f_*Z$ is an effective non-zero $\dim Z$-dimensional irreducible cycle whose support is $f(Z)$, and therefore $[\mu]^{\dim Z}\cdot f(Z)=0$, giving $f(Z)\subset\operatorname{Null}([\mu])$. In this case then, $x\in f^{-1}\bigl(\operatorname{Null}([\mu])\bigr)$.

It remains to show that if $x\in f^{-1}\bigl(\operatorname{Null}([\mu])\bigr)\setminus\operatorname{Exc}(f)$, it follows that $x\in\operatorname{Null}([f^*\mu])$. If $f(x)\in\operatorname{Null}([\mu])$, let $f(x)\in W\subset Y$ be an irreducible positive dimensional, say $\dim W=d$, subvariety such that $[\mu]^{d}\cdot W=0$. Since $x\not\in\operatorname{Exc}(f)$ we can find an irreducible subvariety $x\in W'\subset X$ such that $f|_{W'}\colon W'\to W$ is surjective and generically finite and this gives that $[f^*\mu]^{d}\cdot W'$ is a positive multiple of $[\mu]^d\cdot W$. But then, $x\in \operatorname{Null}([f^*\mu])$.
\end{proof}

\begin{rem}\label{rem:degsetgenfiniteisop1}
By the proof of our Main Theorem and Remark \ref{rem:genfin}, if we have moreover that $f_*\colon\pi_1(X)\to\pi_1(Y)$ is an isomorphism, then very class $\alpha\in\mathcal W_X$ is indeed the pull-back of a class in $\mathcal W_Y$. Thus, we obtain in this case, by taking the intersection over all the $[\mu]\in\mathcal W_Y$ that 
$$
Z_X=f^{-1}(Z_Y)\cup\operatorname{Exc}(f).
$$
\end{rem}

\begin{lem}\label{lem:degeneracyetalecover}
Let $f\colon X\to Y$ be a finite étale cover between weakly Kähler hyperbolic manifolds. 
Then, $Z_X=f^{-1}(Z_Y)$.
\end{lem}

\begin{proof}
Suppose first that $f$ is a Galois cover. By Lemma \ref{lem:degbelowcontainedindegabove} we know that, for $[\mu]\in\mathcal W_Y$, we have $\operatorname{Null}([f^*\mu])=f^{-1}\bigl(\operatorname{Null}([\mu])\bigr)$, since $\operatorname{Exc}(f)=\emptyset$. Therefore, $Z_X\subset\operatorname{Null}([f^*\mu])=f^{-1}\bigl(\operatorname{Null}([\mu])\bigr)$, and taking the intersection over all $[\mu]\in\mathcal W_Y$, we get $Z_X\subset f^{-1}(Z_Y)$.

Next, take $[\alpha_0]\in\mathcal W_X$ such that $Z_X=\operatorname{Null}([\alpha_0])$. Take the average $[\tilde\alpha_0]:=\sum_{\gamma\in G}[\gamma^*\alpha_0]$ of $[\alpha_0]$ with respect to the deck transformation group $G=\operatorname{Aut}(f\colon X\to Y)$. Then, $[\tilde\alpha_0]$ is a sum of nef classes which is $G$-invariant and so on the one hand $[\tilde\alpha_0]=[f^*\mu_0]$ for some $[\mu_0]\in\mathcal W_Y$ being $f\colon X\to Y$ Galois, and on the other hand for each $Z\subset X$ positive dimensional irreducible subvariety we have that $[\tilde\alpha_0]^{\dim Z}\cdot Z$ is a sum of non-negative terms one of which is $[\alpha_0]^{\dim Z}$.

We deduce that if $[\tilde\alpha_0]^{\dim Z}\cdot Z=0$, \textsl{i.e.} if $Z\subset\operatorname{Null}([\tilde\alpha_0])$, then $[\alpha_0]^{\dim Z}\cdot Z=0$ and so $Z\subset\operatorname{Null}([\alpha_0])=Z_X$. In other words, $Z_X\subset\operatorname{Null}([\tilde\alpha_0])\subset\operatorname{Null}([\alpha_0])=Z_X$ and thus $Z_X=\operatorname{Null}([\tilde\alpha_0])=\operatorname{Null}([f^*\mu_0])$.

But, $\operatorname{Null}([f^*\mu_0])=f^{-1}\bigl(\operatorname{Null}([\mu_0])\bigr)\supset f^{-1}(Z_Y)$, so that $f^{-1}(Z_Y)\subset Z_X$ and we are done.

\smallskip

In the general case where $f\colon X\to Y$ is not necessarily Galois, take a finite Galois covering $h\colon \hat X\to Y$ factoring through $X$ via a finite Galois covering $g\colon\hat X\to X$ (this is always possible, see for instance the argument in the proof of \cite[Proposition 2.6]{BDET24}). Then, we have 
$$
Z_{\hat X}=g^{-1}(Z_X)=h^{-1}(Z_Y)=g^{-1}\bigl(f^{-1}(Z_Y)\bigr),
$$
and so $Z_X=f^{-1}(Z_Y)$.
\end{proof}

We can now give the proof of Corollary \ref{cor:cor3}. We proceed as in the proof of the second part of Corollary \ref{cor:cor2}, in particular considering the diagram \eqref{diag:genfin}. The situation is now as follows:
$$
\xymatrix{
 & & W \ar[ddl]^\pi \\
 X\ar[dr]^f \ar[urr]^{\tilde f} & &\\
 &Y, &
}
$$
where $\pi\colon W\to Y$ is the finite étale covering corresponding to the subgroup $\pi_1(Y)\supset f_*(\pi_1(X))\simeq\pi_1(X)$, and $\tilde f$ is a lifting of $f$. As observed during the proof of Corollary \ref{cor:cor3}, $\tilde f$ is generically finite and surjective, and induces an isomorphism on fundamental groups. By Remark \ref{rem:degsetgenfiniteisop1}, we have $Z_X=\tilde f^{-1}(Z_W)\cup\operatorname{Exc}(\tilde f)$. By Lemma \ref{lem:degeneracyetalecover}, we have $Z_W=\pi^{-1}(Z_Y)$; moreover, $\operatorname{Exc}(f)=\operatorname{Exc}(\tilde f)$.

Putting all together, we obtain 
$$
\begin{aligned}
Z_X & =\tilde f^{-1}\bigl(\pi^{-1}(Z_Y)\bigr)\cup\operatorname{Exc}(\tilde f) \\
& =f^{-1}(Z_Y)\cup\operatorname{Exc}(f).
\end{aligned}
$$
\qed

\section{Toroidal compactifications of ball quotients}\label{sec:toroidal}

In this section, we exhibit a new family of natural examples of weakly K\"ahler hyperbolic manifolds. To do so, let $G<\PU(n,1)$ be a torsion free lattice which is not cocompact. The manifold $X^\circ:=G\backslash \B^n$ is thus a quasi-projective variety that can be compactified in various ways. Two important such compactifications are the following (for which we refer the reader to \cite{AMRT10,Mok12}):
\begin{enumerate}
\item The Baily--Borel compactification $X^\circ\subset X_{\BB}$ is the normal projective variety obtained from $X^\circ$ by adding points (\emph{closing the cusps} from the differential geometric viewpoint). It is minimal in the sense that given any normal compactification $X^\circ\hookrightarrow \bar X$, the identity map on $X$ extends to a holomorphic map $\bar X\to X_{\BB}$.
\item The toroidal compactification $X^\circ\subset X_{\tor}$ is a smooth projective variety; the boundary $X_{\tor}\setminus X^\circ$ is a finite union of disjoint hypersurfaces that are isomorphic to abelian varieties and have negative normal bundle. Their blow-down to isolated normal singularities gives back $X_{\BB}$.
\end{enumerate}

Suitable toroidal compactifications of these quotients provide the examples mentioned in the introduction.

\begin{thm}\label{thm:toroidal-compactification-wkh}
Up to replacing $G$ with a finite index subgroup, the manifold $X_{\tor}$ is weakly K\"ahler hyperbolic. 

Moreover, it is not bimeromorphic to any K\"ahler hyperbolic manifold.
\end{thm}

\begin{proof}
According to \cite[Lemma~3.2]{zhu24}, we know that $X_{\BB}$ is aspherical and that $\pi_1(X_{\BB})$ is Gromov-hyperbolic (this is only achieved after replacing $G$ with a finite index subgroup). Asphericity implies that
$$
H^2(X_{\BB},\R)=H^2\bigl(\pi_1(X_{\BB}),\R\bigr),
$$
and according to Lemma~\ref{lem:bounded cohomology} and Gromov-hyperbolicity of $\pi_1(X_{\BB})$, we infer that
$$
H^2(X_{\BB},\R)=V^2_{\hyp}(X_{\BB}).
$$
Consider an embedding $\iota\colon X_{\BB}\hookrightarrow \PP^N$ and the contraction of the exceptional tori $\pi\colon X_{\tor}\to X_{\BB}$. Since any class in $H^2(X_{\BB},\R)$ is hyperbolic, we can pick an ample class $\alpha$ in $\mathbb P^N$ and pull it back via $\iota\circ\pi$ to $X_{\tor}$.
This class $\omega=(\iota\circ\pi)^*\alpha$ is obviously big and nef and it is also hyperbolic since $\iota^*\alpha$ is so.

Let us now check that $X_{\tor}$ is not birational to any Kähler hyperbolic manifold. By contradiction, suppose $\mu\colon X_{\tor}\dasharrow Z$ is a bimeromorphic mapping, with $Z$ a Kähler hyperbolic manifold. Since $X_{\tor}$ is smooth and $Z$ does not contain any rational curve being Kähler hyperbolic, by \cite[Corollary 1.44]{Deb01}, $\mu$ is indeed everywhere defined, \textsl{i.e.} it is a modification. Through a general point of the exceptional locus of $\mu$ there is a rational curve (which is contracted by $\mu$), \cite[Proposition 1.43]{Deb01}; any rational curve in $X_{\tor}$ is not contained in the exceptional locus of $\pi\colon X_{\tor}\to X_{\BB}$ since it consists of tori. Now, for the class $\omega$ above the null locus coincide with the union of the $\pi$-exceptional tori, and any rational curve must be contained in this null locus. The upshot is that $X_{\tor}$ cannot contain any rational curve and $\mu$ must be a biholomorphism, but this is impossible since $X_{\tor}$ contains tori, and $Z$ does not, being Kähler hyperbolic.
\end{proof}

\begin{rem}
Observe that by \cite[Corollary 3]{sarem23}, $X_{\tor}$ has a Kähler metric with non-positive holomorphic sectional curvature, which can be used in the proof above to exclude directly the existence of rational curves in $X_{\tor}$.
\end{rem}

\bibliography{bibliography}{}
	
\end{document}